\documentclass[11pt,a4j]{article}

\usepackage{lineno,hyperref}
 \usepackage{amsmath,amssymb,amsfonts}
 \usepackage{amsthm}
 \usepackage{latexsym}
 \usepackage{array}
 \usepackage{color}
 
 \theoremstyle{definition}
\newtheorem{theorem}{Theorem}[section]
\newtheorem{corollary}[theorem]{Corollary}
\newtheorem{prop}[theorem]{Proposition}

\newtheorem{rem}{Remark}
\newtheorem{example}{Example}
\newtheorem{defn}{Definition}

\allowdisplaybreaks[4] 

\title{Compatibility between Jacobi structures and pseudo-Riemannian cometrics on Jacobi algebroids}

\author{Naoki Kimura \thanks{Department of Applied Mathematics, School of Fundamental Science and Engineering, Waseda University, \textup{3-4-1} Okubo, Shinjuku, Tokyo \textup{169-8555}, Japan} and Tomoya Nakamura  \thanks{Academic Support Center, Kogakuin University, \textup{2665-1}, Nakano-cho, Hachioji-shi, Tokyo, Japan}
 \\ email: 
 \href{mailto:noverevitheuskyk@toki.waseda.jp}
 {noverevitheuskyk@toki.waseda.jp}\\
 \href{mailto:kt13676@ns.kogakuin.ac.jp}
 {kt13676@ns.kogakuin.ac.jp}}
\date{}

\begin{document}

\maketitle

\begin{abstract}
We define compatibility between Jacobi structures and pseudo-Riemannian cometrics on Jacobi algebroids.
This notion is a generalization of the compatibility between Poisson structures and pseudo-Riemannian cometrics on manifolds, which was defined by Boucetta \cite{Bou1}.
We show that the compatibility with a cometric is ``preserved'' by the Poissonization of a Jacobi structure.
Furthermore, we prove that for a contact pseudo-metric structure on a manifold, satisfying the compatibility condition is equivalent to being a Sasakian pseudo-metric structure.


  \end{abstract}




\section{Introduction}
Jacobi manifolds were introduced by Lichnerowicz and Kirillov independently as a generalization of Poisson manifolds.
A Jacobi manifold is also a generalization of a contact manifold.  
The Poissonization of a Jacobi structure on a manifold $M$ is an operation which gives a Poisson structure on the manifold $M \times \mathbb{R}$. 
The obtained Poisson structure on $M \times \mathbb{R}$ is also called the Poissonization of a given Jacobi structure on $M$.
For a contact manifold, the Poissonization is an equivalent operation to the symplectization of a given contact manifold.
The Poissonization plays a central role in the study of Jacobi manifolds since Poisson manifolds are less complicated in various aspects than Jacobi manifolds.
Notice that the Poissonization extends to a Jacobi structure on a Jacobi algebroid, which is a generalization of a Jacobi structure on a manifold.

Boucetta \cite{Bou1} defined compatibility between Poisson structures and pseudo-Riemannian cometrics on manifolds by using an affine connection on the cotangent bundle.
He showed that if a non-degenerate Poisson structure has a compatible cometric, the corresponding symplectic form is a K\"{a}hler form.
Due to this result, a Poisson structure with a compatible cometric is considered as a generalization of a K\"{a}hler structure.
The compatibility between Poisson structures and cometrics have been extensively studied. 
For instance, the case for the Lie-Poisson structure on the dual space of a Lie algebra is studied in \cite{Bou2} \cite{Chen}.

In this paper, we define compatibility between Jacobi structures and pseudo-Riemannian cometrics on Jacobi algebroids.
This notion is a generalization of the compatibility between Poisson structures and pseudo-Riemannian cometrics on manifolds.
Compatibility between Jacobi structures and pseudo-Riemannian cometrics on manifolds was already defined in \cite{AZ1}. However that is different from the definition in this paper. 
In their definition, the compatibility is defined by using the cotangent bundle with a Lie algebroid structure associated with a Jacobi structure.
Meanwhile, in our definition, that is defined by using the Whitney sum of the cotangent bundle and the trivial line bundle with the standard Jacobi algebroid structure. In terms of the Poissonization of a Jacobi structure, it is more natural to consider our definition than theirs. In fact, 
we show that the compatibility with a cometric is ``preserved" under the Poissonization of a Jacobi structure.
Furthermore, we prove that for a contact pseudo-metric structure on a manifold, satisfying the compatibility condition is equivalent to being a Sasakian pseudo-metric structure. 
Therefore, a Jacobi structure with a compatible cometric is considered as a generalization of a Sasakian structure.

This paper is organized as follows.
In Section 2, we review the definitions of several notions such as Lie algebroids, Poisson structures, Jacobi algebroids and Jacobi structures. 
In addition, we explain the Poissonization of a Jacobi structure.
In Section 3, we recall the compatibility between Poisson structures and pseudo-Riemannian cometrics defined by Boucetta \cite{Bou1}. 
After that, as a generalization of that notion, we define compatibility between Jacobi structures and pseudo-Riemannian cometrics on Jacobi algebroids.
We show that the compatibility with a cometric is ``preserved" under the Poissonization of a Jacobi structure.
At the end, we state that a Sasakian pseudo-metric structure is regarded as a special case of a Jacobi structure with a compatible cometric. 

\section*{Acknowledgement}
The first author is supported by a Waseda University Grant for Special Research Projects (Project number: 2022C-434).

\section{Preliminaries}\label{Preliminaries}

In this section, we recall the definitions and properties of Lie algebroids, Poisson structures, Jacobi algebroids and Jacobi structures. See \cite{IM} for details on Jacobi algebroids and Jacobi structures.

\subsection{Lie algebroids and Poisson structures}\label{Lie algebroids and Poisson structures}

A {\it skew algebroid} over a manifold $M$ is a vector bundle
$A\rightarrow M$ equipped with a skew symmetric $\mathbb{R}$-bilinear map $[\cdot,\cdot]_{A}:\Gamma (A)\times \Gamma(A)\rightarrow \Gamma (A)$, called the {\it bracket}, and a bundle map $\rho_{A}:A\rightarrow TM$ over $M$, called the {\it anchor}, satisfying the following condition: for any $X,Y$ in $\Gamma (A)$ and $f$ in $C^{\infty }(M)$,
 \begin{align*}
[X,fY]_{A}=f[X,Y]_{A}+(\rho_{A}(X)f)Y,
 \end{align*}
where we denote the map $\Gamma(A)\rightarrow \Gamma(TM)=\mathfrak{X}(M)$ induced by the anchor, the same symbol $\rho_{A}$. A {\it Lie algebroid} over a manifold $M$ is a skew algebroid $(A,[\cdot,\cdot]_A,\rho_A)$ such that the bracket satisfies the Jacobi identity, i.e., $[\cdot,\cdot]_A$ is a Lie bracket on $\Gamma (A)$. For any Lie algebroid $(A,[\cdot,\cdot]_A,\rho_A)$ over $M$, it follows that for any $X$ and $Y$ in $\Gamma(A)$,
 \begin{align*}
\rho_A([X,Y]_{A})=[\rho_A(X),\rho_A(Y)],
 \end{align*}
where the bracket on the right hand side is the usual Lie bracket on $\mathfrak{X}(M)$.

\begin{example}

For any manifold $M$, the tangent bundle $(TM,[\cdot,\cdot],\mbox{id}_{TM})$ is a Lie algebroid over $M$, where $[\cdot,\cdot]$ is the usual Lie bracket on the vector fields $\mathfrak{X}(M)=\Gamma(TM)$. 
%
\end{example}

Let $(A,[\cdot,\cdot]_A,\rho_A)$ be a skew algebroid over $M$. The {\it Schouten bracket} on $\Gamma (\Lambda ^*A)$ is defined similarly to the Schouten bracket $[\cdot ,\cdot ]$ on the multivector fields $\mathfrak{X}^*(M)$. That is, the Schouten bracket $[\cdot ,\cdot ]_A:\Gamma (\Lambda ^kA)\times \Gamma (\Lambda ^lA)\rightarrow \Gamma (\Lambda ^{k+l-1}A)$ is defined as the unique extension of the bracket $[\cdot ,\cdot ]_A$ on $\Gamma (A)$ such that
\begin{align*}
&[f,g]_A=0;\\
&[X,f]_A=\rho_A(X)f;\\
&[X,Y]_A\ \mbox{is the bracket on}\ \Gamma (A);\\
&[D_1,D_2\wedge D_3]_A=[D_1,D_2]_A\wedge D_3+(-1)^{\left(a_1+1\right)a_2}D_2\wedge [D_1, D_3]_A;\\
&[D_1,D_2]_A=-(-1)^{(a_1-1)(a_2-1)}[D_2,D_1]_A
\end{align*}
for any $f,g$ in $C^\infty(M)$, $X,Y$ in $\Gamma (A)$, $D_i$ in $\Gamma (\Lambda ^{a_{i}}A)$. The {\it differential} of the skew algebroid $A$ is an operator $d_A:\Gamma (\Lambda ^k A^*)\rightarrow \Gamma (\Lambda ^{k+1} A^*)$ defined by for any $\omega $ in $\Gamma (\Lambda^k A^*)$ and $X_0,\dots ,X_k$ in $\Gamma(A)$,
\begin{align}\label{A-gaibibun}
(d_A \omega)(X_0, \dots ,X_k)&=\sum_{i=0}^k(-1)^i\rho_A (X_i)(\omega(X_0,\dots ,\hat{X_i},\dots ,X_k))\nonumber \\
&\quad+\sum _{i<j}^{}(-1)^{i+j}\omega ([X_i,X_j]_A , X_0,\dots ,\hat{X_i},\dots ,\hat{X_j},\dots ,X_k).
\end{align}
If $(A,[\cdot,\cdot]_A,\rho_A)$ is a Lie algebroid, $d_A^2=0$ holds. For any $X$ in $\Gamma (A)$, the {\it Lie derivative\index{Lie derivative}} $\mathcal{L}_X^A:\Gamma (\Lambda ^k A^*)\rightarrow \Gamma (\Lambda ^k A^*)$ is defined by the {\it Cartan formula} $\mathcal{L}_X^A:=d_A \iota _X +\iota _X d_A$ and $\mathcal{L}_X^A$ are extended on $\Gamma (\Lambda ^*A)$ in the same way as the usual Lie derivative $\mathcal{L}_X$ respectively. Then it follows that $\mathcal{L}_X^AD=[X,D]_A$ for any $D$ in $\Gamma (\Lambda^{*}A)$. We call a $d_{A}$-closed $2$-cosection $\omega$, i.e., $d_{A}\omega=0$, a {\it presymplectic structure} on $(A,[\cdot,\cdot]_A,\rho_A)$. A presymplectic structure $\omega$ is called a {\it symplectic structure} if $\omega$ is non-degenerate.

\begin{rem}
In the definition of the Schouten bracket, some authors use a condition 
\begin{align}\label{schouten betsu def}
[D_1,D_2]_A=(-1)^{a_1a_2}[D_2,D_1]_A
\end{align}
for any $D_i$ in $\Gamma (\Lambda ^{a_{i}}A)$ instead of the condition $[D_1,D_2]_A=-(-1)^{(a_1-1)(a_2-1)}$ $[D_2,D_1]_A$.
\end{rem}

\begin{example}\label{A oplus R}
Let $A$ be a vector bundle over a manifold $M$ and set $A\oplus \mathbb{R}:=A\oplus(M\times \mathbb{R})$. Then the sections $\Gamma (\Lambda^k(A\oplus \mathbb{R}))$ and $\Gamma (\Lambda^k(A\oplus \mathbb{R})^*)$ can be identified with $\Gamma (\Lambda^kA)\times \Gamma (\Lambda^{k-1}A)$ and $\Gamma (\Lambda^kA^*)\times \Gamma (\Lambda^{k-1}A^*)$ as follows:
\begin{align}
(P,Q)((\alpha_1,f_1),&\dots,(\alpha_k,f_k))\nonumber\\
                     &=P(\alpha_1,\dots,\alpha_k)+\sum_i(-1)^{i+1}f_iQ(\alpha_1,\dots,\hat{\alpha}_i,\dots,\alpha_k),\label{multi-vector field formula}\\
(\alpha,\beta)((X_1,f_1),&\dots,(X_k,f_k))\nonumber\\
                         &=\alpha(X_1,\dots,X_k)+\sum_i(-1)^{i+1}f_i\beta(X_1,\dots,\hat{X}_i,\dots,X_k) \label{differential form formula}
\end{align}
for any $(P,Q)$ in $\Gamma (\Lambda^kA)\times \Gamma (\Lambda^{k-1}A)$, $(\alpha,\beta)$ in $\Gamma (\Lambda^kA^*)\times \Gamma (\Lambda^{k-1}A^*)$, $(\alpha _i,f_i)$ in $\Gamma (A^*)\times C^\infty(M)$ and $(X_i,f_i)$ in $\Gamma (A)\times C^\infty(M)$. Moreover under the identifications, the exterior products are given by
\begin{align*}
(P_1,Q_1)\wedge(P_2,Q_2)&=(P_1\wedge P_2,Q_1\wedge P_2+(-1)^{a_1}P_1\wedge Q_2),\\
(\alpha_1,\beta_1)\wedge(\alpha_2,\beta_2)&=(\alpha_1\wedge \alpha_2,\beta_1\wedge \alpha_2+(-1)^{a_1}\alpha_1\wedge \beta_2)
\end{align*}
for any $(P_i,Q_i)$ in $\Gamma (\Lambda^{a_i}A)\times \Gamma (\Lambda^{a_i-1}A)$ and $(\alpha_i,\beta_i)$ in $\Gamma (\Lambda^{a_i}A^*)\times \Gamma (\Lambda^{a_i-1}A^*)$. Now, assume that $A$ is a skew (resp. Lie) algebroid over $M$. Then $(A\oplus \mathbb{R},[\cdot,\cdot]_{A\oplus\mathbb{R}},\rho_{A}\circ\mbox{pr}_1)$ is also a skew (resp. Lie) algebroid over $M$, where the bracket $[\cdot,\cdot]_{A\oplus\mathbb{R}}$ is defined by
\begin{align}\label{A oplus R no kakko}
[(X,f),(Y,g)]_{A\oplus\mathbb{R}}&:=([X,Y]_{A},\rho_{A}(X)g-\rho_{A}(Y)f)
\end{align}
and the map $\mbox{pr}_1:A\oplus \mathbb{R}\rightarrow A$ is the canonical projection to the first factor. In this case, the differential $d_{A\oplus \mathbb{R}}$ of the skew (resp. Lie) algebroid $A\oplus \mathbb{R}$ and the Schouten bracket $[\cdot,\cdot]_{A\oplus \mathbb{R}}$ are given by
\begin{align*}
d_{A\oplus \mathbb{R}}(\alpha,\beta)&=(d_{A}\alpha,-d_{A}\beta),\\
[(P_1,Q_1),(P_2,Q_2)]_{A\oplus\mathbb{R}}&=([P_1,P_2]_A,(-1)^{k+1}[P_1,Q_2]_A-[Q_1,P_2]_A)
\end{align*}
for any $(\alpha,\beta)$ in $\Gamma (\Lambda^{k}A^{*})\times \Gamma (\Lambda^{k-1}A^{*})$ and $(P_i,Q_i)$ in $\Gamma (\Lambda^{k}A)\times \Gamma (\Lambda^{k-1}A)$. 
\end{example}

{\it A Poisson structure} on a skew (resp. Lie) algebroid $A$ over a manifold $M$ is a $2$-section $\pi$ in $\Gamma(\Lambda^2 A)$ satisfying $[\pi,\pi]_A=0$. For any $2$-section $\pi $ in $\Gamma (\Lambda^2A)$, we define a skew-symmetric bilinear bracket $[\cdot,\cdot]_{\pi} $ on $\Gamma (A^*)$ by for any $\xi,\eta$ in $\Gamma (A^*)$,
\begin{align}
[\xi,\eta]_{\pi}:=\mathcal{L}_{\pi^\sharp \xi}^{A}\eta -\mathcal{L}_{\pi^\sharp \eta}^{A}\xi -d_{A}\langle \pi^\sharp \xi,\eta \rangle,
\end{align}
where a bundle map $\pi^\sharp:A^{*}\rightarrow A$ over $M$ is defined by $\langle \pi^\sharp \xi,\eta\rangle:=\pi(\xi,\eta)$. Then a triple $(A^*,[\cdot,\cdot]_\pi,\rho_\pi)$, where $\rho_\pi:=\rho_A\circ\pi^\sharp$, is a skew algebroid. We denote $(A^*,[\cdot,\cdot]_\pi,\rho_\pi)$ by $A^*_\pi$ and the differential of $A^*_\pi$ by $d_\pi$. Then $d_\pi D=[\pi,D]_A$ holds for any $D$ in $\Gamma(\Lambda^*A)$. Moreover it follows that
 \begin{align}\label{Jacobi structure bracket property}
\frac{1}{2}[\pi ,\pi ]_{A}(\xi,\eta,\cdot)=[\pi^\sharp \xi,\pi^\sharp \eta ]_{A}-\pi^\sharp[\xi,\eta]_{\pi}.
\end{align}
In the case that $(A, [\cdot,\cdot]_A,\rho_A)$ is a Lie algebroid, a skew algebroid $A^*_\pi$ is a Lie algebroid if and only if $\pi$ is Poisson. 

It is well known that there exists a one-to-one correspondence between symplectic structures and non-degenerate Poisson structures on a skew algebroid $(A, [\cdot,\cdot]_A,\rho_A)$. In fact, for a non-degenerate Poisson structure $\pi$, a $2$-cosection $\omega_{\pi}$ characterized by $\omega_{\pi}^{\flat}=-(\pi^{\sharp})^{-1}$ is symplectic, where for any $2$-cosection $\Omega$, a bundle map $\Omega^\flat:A\rightarrow A^{*}$ over $M$ is defined by $\langle \Omega^\flat X,Y\rangle:=\Omega(X,Y)$ for any $X$ and $Y$ in $\Gamma(A)$.

\subsection{Jacobi algebroids and Jacobi structures}\label{Jacobi algebroids and Jacobi structures}

A pair $(A,\phi_0)$ is a {\it Jacobi algebroid} over a manifold $M$ if $A=(A,[\cdot ,\cdot ]_A,\rho_A)$ is a Lie algebroid over $M$ and $\phi_0$ in $\Gamma(A^*)$ is $d_A$-closed, that is, $d_A\phi_{0}=0$. 

\begin{example}\label{trivial example}
For any Lie algebroid $A$ over $M$, we set $\phi_0:=0$. Then $(A,\phi_0)$ is a Jacobi algebroid. We call $\phi_0$ the {\it trivial Jacobi algebroid structure} on $A$. Therefore any Lie algebroid is a Jacobi algebroid.
\end{example}

\begin{example}
For a Lie algebroid $A\oplus \mathbb{R}$ in Example \ref{A oplus R}, We set $\phi_0:=(0,1)$ in $\Gamma(A^{*}\oplus \mathbb{R})=\Gamma (A^{*})\times C^{\infty}(M)$. Then $(A\oplus \mathbb{R},\phi_0)$ is a Jacobi algebroid.
\end{example}

For a Jacobi algebroid $(A,\phi_{0})$, there is the {\it $\phi_0$-Schouten bracket} $[\cdot,\cdot]_{A,\phi_0}$ on $\Gamma (\Lambda^*A)$ given by
\begin{align}\label{phi-bracket}
[D_1,D_2]_{A,\phi_0}:=[D_1,D_2]_A+(a_1-1)&D_1\wedge \iota _{\phi_0}D_2 \nonumber\\
                                                     &-(-1)^{a_1+1}(a_2-1)\iota_{\phi_0}D_1\wedge D_2
\end{align}
for any $D_i$ in $\Gamma(\Lambda ^{a_i}A)$, where $[\cdot,\cdot]_A$ is the Schouten bracket of the Lie algebroid $A$. The {\it $\phi_0$-differential} $d_{A,\phi_0}$ and the {\it $\phi_0$-Lie derivative} $\mathcal{L}_X^{A,\phi_0}$ are defined by
\begin{align*}
d_{A,\phi_0}\omega :=d_A\omega +\phi_0\wedge \omega,\quad \mathcal{L}_X^{A,\phi_0}:=\iota _X\circ d_{A,\phi_0}+d_{A,\phi_0}\circ \iota_X
\end{align*}
for any $\omega$ in $\Gamma (\Lambda ^*A^*)$ and $X$ in $\Gamma (A)$.
 
 We notice that
\begin{align*}
(d_{A,\phi_0}\omega )(X_0,&\dots,X_k)\\
                          &=\sum_i(-1)^{i+1}\rho_{A,\phi_0}(X_i)\omega(X_0,\dots,\hat{X}_i,\dots,X_k)\\
                          &\quad +\sum_{i<j}(-1)^{i+j}\omega([X_i,X_j]_A,X_0,\dots,\hat{X}_i,\dots,\hat{X}_j,\dots,X_k)
\end{align*}
for any $\omega$ in $\Gamma (\Lambda^kA^*)$ and $X_i$ in $\Gamma (A)$, and that
\begin{align*}
\mathcal{L}_X^{A,\phi_0}\omega=\mathcal{L}_X^A\omega+\langle\phi_0,X\rangle\omega
\end{align*}
for any $\omega$ in $\Gamma (\Lambda ^*A^*)$ and $X$ in $\Gamma (A)$. Here $\rho_{A,\phi_0}(X)f:=\rho_A(X)f+\langle \phi_0,X\rangle f$ for any $X$ in $\Gamma (A)$ and $f$ in $C^\infty(M)$. We call a $d_{A,\phi_0}$-closed $2$-cosection $\omega$, i.e., $d_{A,\phi_0}\omega=0$, a {\it $\phi_0$-presymplectic structure} on $(A,\phi_{0})$. A $\phi_0$-presymplectic structure $\omega$ is called a {\it $\phi_0$-symplectic structure} if $\omega$ is non-degenerate.

\begin{rem}
In the case using the condition (\ref{schouten betsu def}) in the definition of the Schouten bracket $[\cdot,\cdot]_A$, the $\phi_0$-Schouten bracket $[\cdot,\cdot]_{A,\phi_0}$ is given by
\begin{align*}
[D_1,D_2]_{A,\phi_0}:=[D_1,D_2]_A+(-1)^{a_1+1}(a_1-1)&D_1\wedge \iota _{\phi_0}D_2 \\
&-(a_2-1)\iota_{\phi_0}D_1\wedge D_2
\end{align*}
instead of (\ref{phi-bracket}).
\end{rem}

 \begin{example}\label{contact example}
We consider a Jacobi algebroid $(A,\phi_0)$ over $M$, where $A:=TM\oplus\mathbb{R}$ and $\phi_{0}:=(0,1)$ in $\Omega^{1}(M)\times C^{\infty}(M)$. Then any $\omega $ in $\Omega^{2}(M)\times \Omega^{1}(M)$ can be written as $\omega=(\alpha,\beta)\ (\alpha\in\Omega^{2}(M),\beta\in\Omega^{1}(M))$. Since
 \begin{align*}
d_{A,\phi_{0}}\omega=d_{TM\oplus\mathbb{R},(0,1)}(\alpha,\beta)=(d\alpha,\alpha-d\beta),
\end{align*}
$\omega$ is $(0,1)$-presymplectic on $(TM\oplus\mathbb{R},(0,1))$ if and only if $\omega =(d\beta,\beta)\ (\beta\in\Omega^{1}(M))$. Moreover setting $\dim M=2n+1$, we see that a $(0,1)$-presymplectic strucutre $\omega $ is non-degenerate if and only if $\beta\wedge (d\beta)^{n}\neq 0$, that is, $\beta$ is a {\it contact structure} on $M$. Therefore a $(0,1)$-symplectic structure on $(TM\oplus \mathbb{R},(0,1))$ is just a contact structure on $M$.
 \end{example}

A {\it Jacobi structure} on a Jacobi algebroid $(A,\phi_{0})$ is a $2$-section $\pi $ in $\Gamma (\Lambda^{2}A)$ satisfying the condition
\begin{equation}\label{Jacobi def equation}
[\pi,\pi]_{A,\phi_{0}}=0.
\end{equation}
For any $2$-section $\pi $ on $(A,\phi_0)$, we define a skew-symmetric bilinear bracket $[\cdot,\cdot]_{\pi,\phi_{0}} $ on $\Gamma (A^*)$ by for any $\xi,\eta$ in $\Gamma (A^*)$,
\begin{align}
[\xi,\eta]_{\pi,\phi_{0}}:=\mathcal{L}_{\pi^\sharp \xi}^{A,\phi_0}\eta -\mathcal{L}_{\pi^\sharp \eta}^{A,\phi_0}\xi -d_{A,\phi_0}\langle \pi^\sharp \xi,\eta \rangle.
\end{align}
Then a triple $(A^*,[\cdot,\cdot]_{\pi,\phi_{0}},\rho_\pi)$, where $\rho_\pi:=\rho_{A} \circ \pi^{\sharp}$, is a skew algebroid. 
Moreover it follows that
 \begin{align}\label{Jacobi structure bracket property}
\frac{1}{2}[\pi ,\pi ]_{A,\phi_0}(\xi,\eta,\cdot)=[\pi^\sharp \xi,\pi^\sharp \eta ]_{A}-\pi^\sharp[\xi,\eta]_{\pi,\phi_{0}}.
\end{align}
Then $A^*_{\pi,\phi_0} := (A^*, [\cdot,\cdot]_{\pi,\phi_0}, \rho_\pi)$ is a Lie algebroid over $M$ if and only if $\pi$ is Jacobi. Furthermore, in the case that $\pi$ is Jacobi, a pair $(A^*_{\pi,\phi_0}, X_0)$ is a Jacobi algebroid over $M$, where $X_0:=-\pi^{\sharp}\phi_0$ in $\Gamma(A)$. We call it {\it the Jacobi algebroid induced by a Jacobi structure $\pi$ on $(A,\phi_0)$}.

\begin{example}[Poisson structures]
For any Lie algebroid $A$ equipped with the trivial Jacobi algebroid structure $0$, it follows that $[\cdot,\cdot]_{A,0}=[\cdot,\cdot]_{A}$. Hence Jacobi structures on $(A,0)$ are just Poisson structures on $A$. In this case, the Lie algebroid $A^*_{\pi,0}$ induced by a Jacobi structure $\pi$ on $(A,0)$ coincides with the Lie algebroid $A^*_{\pi}$ induced by a Poisson structure $\pi$ on $A$.  
\end{example}

\begin{example}\label{Jacobi manifold}
Let $A$ be a Lie algebroid over $M$, $\Lambda $ a $2$-section on $A$ and $E$ a section on $A$ satisfying
\begin{align*}
[\Lambda ,\Lambda ]_{A}=2E\wedge \Lambda ,\quad [E,\Lambda ]_{A}=0.
\end{align*}
Then a pair $(\Lambda ,E)$ in $\Gamma (\Lambda ^{2}A)\oplus \Gamma (A)\cong \Gamma (\Lambda ^{2}(A\oplus \mathbb{R}))$ is a Jacobi structure on a Jacobi algebroid $(A\oplus \mathbb{R},(0,1))$, i.e., it satisfies $[(\Lambda ,E),(\Lambda ,E)]_{A\oplus \mathbb{R},(0,1)}=0$. When $(\Lambda ,E)$ is a Jacobi structure on $(TM\oplus \mathbb{R},(0,1))$, we call it a {\it Jacobi structure on $M$} and a triple $(M,\Lambda,E)$ a {\it Jacobi manifold}. If $\pi $ is a Poisson structure on $A$, Then $(\pi,0)$ is a Jacobi structure on $(A\oplus \mathbb{R},(0,1))$.
\end{example}

It is well known that there exists a one-to-one correspondence between $\phi_{0}$-symplectic structures on $(A,\phi_{0})$ and non-degenerate Jacobi structures on $(A,\phi_{0})$. In fact, for a non-degenerate Jacobi structure $\pi$ on $(A,\phi_{0})$, a $2$-cosection $\omega_{\pi}$ characterized by $\omega_{\pi}^{\flat}=-(\pi^{\sharp})^{-1}$ is $\phi_{0}$-symplectic on $(A,\phi_{0})$. In particular, there exists a one-to-one correspondence between contact structures on $M$ and non-degenerate Jacobi structures on $M$. If $\eta $ is contact on $M$, then $(\Lambda ,E)$ is Jacobi on $M$, where
\begin{align*}
\Lambda(\alpha,\beta)&:=(d\eta)((\eta^\flat)^{-1}(\alpha),(\eta^\flat)^{-1}(\beta))\quad (\alpha ,\beta \in \Omega^1(M)),\\
E&:=\xi.
\end{align*}
Here $\eta^\flat:\mathfrak{X}(M)\rightarrow\Omega^1(M)$ is a linear isomorphism given by
\begin{align*}
\eta^\flat(X):=\iota_Xd\eta +\langle\eta,X\rangle\eta\quad (X\in \mathfrak{X}(M))
\end{align*}
and $\xi$ in $\mathfrak{X}(M)$ is the Reeb vector field of $\eta $.

Let $(A,\phi_0)$ be a Jacobi algebroid over $M$. We set $\tilde{A}:=A\times \mathbb{R}$. Then $\tilde{A}$ is a vector bundle over $M\times \mathbb{R}$. The sections $\Gamma (\tilde{A})$ can be identified with the set of time-dependent sections of $A$. Under this identification, we can define two Lie algebroid structures $([\cdot,\cdot\hat{]}_A^{\phi_0},\hat{\rho}_A^{\phi_0})$ and $([\cdot,\cdot\bar{]}_A^{\phi_0},\bar{\rho}_A^{\phi_0})$ on $\tilde{A}$, where for any $\tilde{X}$ and $\tilde{Y}$ in $\Gamma(\tilde{A})$,
\begin{align}
\label{hat kakko} [\tilde{X},\tilde{Y}\hat{]}_A^{\phi_0}&:=e^{-t}\left([\tilde{X},\tilde{Y}]_A+\langle\phi_0,\tilde{X}\rangle\left(\frac{\partial \tilde{Y}}{\partial t}-\tilde{Y}\right)-\langle\phi_0,\tilde{Y}\rangle\left(\frac{\partial \tilde{X}}{\partial t}-\tilde{X}\right)\right),\\ 
\label{hat anchor} \hat{\rho}_A^{\phi_0}(\tilde{X})&:=e^{-t}\left(\rho_A(\tilde{X})+\langle \phi_0,\tilde{X}\rangle \frac{\partial}{\partial t}\right),\\
\label{bar kakko} [\tilde{X},\tilde{Y}\bar{]}_A^{\phi_0}&:=[\tilde{X},\tilde{Y}]_A+\langle\phi_0,\tilde{X}\rangle\frac{\partial \tilde{Y}}{\partial t}-\langle\phi_0,\tilde{Y}\rangle\frac{\partial \tilde{X}}{\partial t},\\ 
\label{bar anchor} \bar{\rho}_A^{\phi_0}(\tilde{X})&:=\rho_A(\tilde{X})+\langle \phi_0,\tilde{X}\rangle \frac{\partial}{\partial t}.
\end{align}
Conversely, for a Lie algebroid $A$ over $M$ and a section $\phi_0$ on $A$, if the triple $(\tilde{A},[\cdot,\cdot\hat{]}_A^{\phi_0},\hat{\rho}_A^{\phi_0})$ (resp. $(\tilde{A},[\cdot,\cdot\bar{]}_A^{\phi_0},\bar{\rho}_A^{\phi_0})$) defined by (\ref{hat kakko}) and (\ref{hat anchor}) (resp. (\ref{bar kakko}) and (\ref{bar anchor})) is a Lie algebroid over $M\times \mathbb{R}$, then $(A,\phi_0)$ is a Jacobi algebroid over $M$, i.e., $d_{A}\phi_0=0$. A vector bundle $\tilde{A}$ equipped with the Lie algebroid structure $([\cdot,\cdot\hat{]}_A^{\phi_0},\hat{\rho}_A^{\phi_0})$ (resp. $([\cdot,\cdot\bar{]}_A^{\phi_0},\bar{\rho}_A^{\phi_0})$) is denoted by $\tilde{A}_{{\phi}_{0}}^{\wedge}$ (resp. $\tilde{A}_{{\phi}_{0}}^{-}$). Let $\hat{d}_{A}^{\phi_{0}}$ (resp. $\bar{d}_{A}^{\phi_{0}}$) and $\widehat{\mathcal{L}^{A}}^{\phi_{0}}$ (resp. $\overline{\mathcal{L}^{A}}^{\phi_{0}}$) be the differential of $\tilde{A}_{{\phi}_{0}}^{\wedge}$ (resp. $\tilde{A}_{{\phi}_{0}}^{-}$) and the Lie derivative on $\tilde{A}_{{\phi}_{0}}^{\wedge}$ (resp. $\tilde{A}_{{\phi}_{0}}^{-}$), respectively. Then for any $\tilde{f}$ in $C^{\infty}(M\times \mathbb{R})$ and $\tilde{\phi}$ in $\Gamma (\tilde{A})$, the following formulas hold \cite{IM}:
\begin{align}
&\hat{d}_{A}^{\phi_{0}}\tilde{f}=e^{-t}\left(d_{A}\tilde{f}+\frac{\partial \tilde{f}}{\partial t}\phi_{0}\right),\quad \hat{d}_{A}^{\phi_{0}}\tilde{\phi}=e^{-t}\left(d_{A,\phi_{0}}\tilde{\phi}+\phi_{0}\wedge\frac{\partial \tilde{\phi}}{\partial t}\right);\\
&\bar{d}_{A}^{\phi_{0}}\tilde{f}=d_{A}\tilde{f}+\frac{\partial \tilde{f}}{\partial t}\phi_{0},\quad \bar{d}_{A}^{\phi_{0}}\tilde{\phi}=d_{A}\tilde{\phi}+\phi_{0}\wedge\frac{\partial \tilde{\phi}}{\partial t}.
\end{align} 

Let $(A,\phi_0)$ be a Jacobi algebroid over $M$, $\pi$ a 2-section on $A$ and set $\tilde{\pi}:=e^{-t}\pi$ in $\Gamma (\Lambda ^2\tilde{A})$. Then the following holds:
\begin{align}
[\tilde{\pi},\tilde{\pi}\bar{]}_A^{\phi_0}=e^{-2t}[\pi,\pi]_{A,\phi_0}.
\end{align}
Therefore a 2-section $\pi$ on $A$ is a Jacobi structure on a Jacobi algebroid $(A,\phi_0)$ over $M$ if and only if $\tilde{\pi}$ in $\Gamma (\Lambda ^2\tilde{A})$ is a Poisson structure on a Lie algebroid $\tilde{A}_{{\phi}_{0}}^{-}$ over $M\times \mathbb{R}$. The Poisson structure $\tilde{\pi}$ on $\tilde{A}_{{\phi}_{0}}^{-}$ is called {\it the Poissonization} of $\pi$. 

In the case of $(A,\phi_0)=(TM\oplus \mathbb{R}, (0,1))$, the Lie algebroid $\tilde{A}_{{\phi}_{0}}^{-}
$ is isomorphic to the standard Lie algebroid $T(M\times \mathbb{R})$ over $M\times \mathbb{R}$. Then the Poissonization $\widetilde{(\Lambda ,E)}$ of a Jacobi structure $(\Lambda ,E)$ on $(TM\oplus \mathbb{R},(0,1))$ corresponds to a Poisson structure $\Pi:=e^{-t}\left(\Lambda+\frac{\partial}{\partial t}\wedge E\right)$ on $T(M\times \mathbb{R})$. This is just the Poissonizaion of a Jacobi structure on $M$.

\section{Compatibility between Jacobi structures and pseudo-Riemannian cometrics on Jacobi algebroids}\label{The compatibility between Jacobi structures and Riemann metrics on Jacobi algebroids}

\subsection{Compatibility between 2-sections and pseudo-Riemannian cometrics on Lie algebroids}\label{The compatibility between 2-sections and pseudo-Riemann cometrics on Lie algebroids}

An {\it affine connection} on a skew algebroid $(A, [\cdot,\cdot]_{A}, \rho_{A})$ over $M$ is an $\mathbb{R}$-bilinear map 
$\nabla: \Gamma(A)\times \Gamma(A)\rightarrow \Gamma(A)$ satisfying 
for any $f \in C^{\infty}(M)$ and $X ,Y \in \Gamma(A)$, 
\begin{align*}
\nabla_{fX}Y&=f\nabla_X Y, \\
\nabla_X fY&=f\nabla_X Y+(\rho_A(X)f)Y.
\end{align*}

For any pseudo-Riemannian metric $g$ on $A$, 
there exists a unique affine connection $\nabla$ on $(A, [\cdot,\cdot]_{A}, \rho_{A})$ which is torsion-free and compatible with $g$, i.e., it satisfies
\begin{align*}
\nabla_X Y -\nabla_Y X&=[X,Y]_A, \\
\rho_A(X)(g(Y,Z))&=g(\nabla_X Y,Z)+g(Y,\nabla_X Z)
\end{align*}
for any $X,Y$ and $Z \in \Gamma(A)$.
The unique affine connection $\nabla$ on $(A, [\cdot,\cdot]_{A}, \rho_{A})$ is called the {\it Levi-Civita connection} of $g$.  
The Levi-Civita connection $\nabla$ of $g$ on $(A, [\cdot,\cdot]_{A}, \rho_{A})$ is characterized by the Koszul formula$:$ 
\begin{align*}
2g(\nabla_X Y,Z)=&\rho_A(X)(g(Y,Z))+\rho_A(Y)(g(X,Z)) -\rho_A(Z)(g(X,Y)) \\
&-g([Y,Z]_A,X)-g([X,Z]_A,Y)+g([X,Y]_A,Z).
\end{align*}

\begin{defn}\label{skcom}
Let $(A, [\cdot,\cdot]_{A}, \rho_{A})$ be a skew algebroid over $M$, $\pi$ a 2-section on $A$ and $g^*$ a pseudo-Riemannian metric on $A^*$.  
The pair $(\pi,g^*)$ is said to be compatible on $A$ if 
\begin{equation*}
D^\pi \pi =0,
\end{equation*}
i.e.,
\begin{equation*}
(\pi ^{\sharp}\alpha)(\pi(\beta,\gamma))=\pi(D^{\pi}_{\alpha} \beta,\gamma)+\pi(\beta,D^{\pi}_{\alpha}\gamma)
\end{equation*}
for any $\alpha,\beta$ and $\gamma \in \Gamma(A^*)$,
where $D^\pi$ is the Levi-Civita connection of $g^*$ on the skew algebroid $A^*_{\pi}$.
\end{defn}

\begin{prop}\label{always Poisson}
$(A, [\cdot,\cdot]_{A}, \rho_{A})$, $\pi$ and $g^*$ are same in Definition \ref{skcom}.  
If the pair $(\pi,g^*)$ is compatible, then $[\pi,\pi]_A=0$.
\end{prop}

This proposition implies that a 2-section $\pi$ on a skew algebroid $A$ with a compatible cometric is always a Poisson structure on $A$.

Definition \ref{skcom} is a natural extension of the following definition of the compatibility between a Poisson structure on a manifold and a cometric in [Boucetta].

\begin{defn}[Boucetta]\label{Boucetta}
Let $(M,\pi)$ be a Poisson manifold and $g^*$ a pseudo-Riemannian metric on $T^*M$.  
The pair $(\pi,g^*)$ is said to be compatible on $M$ if 
\begin{equation*}
D^{\pi}\pi =0,
\end{equation*}
where $D^{\pi}$ is the Levi-Civita connection of $g^*$ on the Lie algebroid $(T^* M)_\pi$.
\end{defn}
\begin{rem}
If $(\pi,g^*)$ is compatible on $M$ and $\pi$ is non-degenerate, then the corresponding symplectic form $\omega$ to $\pi$ is a K\"ahler form.
Hence a Poisson structure with a compatible cometric is considered as a generalization of a K\"ahler structure.
\end{rem}

\subsection{Compatibility between 2-sections and pseudo-Riemannian cometrics on Jacobi algebroids}\label{The compatibility between 2-sections and pseudo-Riemann cometrics on Jacobi algebroids}

In this subsection, we shall define compatibility between 2-sections and pseudo-Riemannian cometrics on Jacobi algebroids and investigate their properties. Although A{\"i}t Amrane and Zeglaoui \cite{AZ1}\cite{AZ2} defined compatibility of Jacobi structures and pseudo-Riemannian metrics on manifolds,
their definition is different from the following one.

\begin{defn}
Let $(A,\phi_0)$ be a Jacobi algebroid over $M$, $\pi $ a 2-section on $(A,\phi_0)$ and $g^*$ a pseudo-Riemannian metric on $A^*$. The pair $(\pi,g^*)$ is said to be {\it compatible} on $(A,\phi_0)$ if 
\begin{align*}
(D^{\pi,\phi_0}_{\alpha} \pi)(\beta,\gamma)=-\frac{1}{2}((&X_0\otimes \pi)(\beta,\gamma,\alpha)+(X_0\otimes \pi)(\gamma,\alpha,\beta) \\
&+g^*(\alpha,\beta)\pi((g^*)^{\flat-1}(X_0),\gamma) \\
&-g^*(\alpha,\gamma)\pi((g^*)^{\flat-1}(X_0),\beta)),
\end{align*}
where $D^{\pi,\phi_0}$ is the Levi-Civita connection of $g^*$ on the skew algebroid $A^*_{\pi,\phi_0}$ induced by $\pi$.
\end{defn}



\begin{rem}
If $\phi_0=0$, the above definition is equivalent to the compatibility of $(\pi,g^*)$ on a Lie algebroid $A$ (See Definition \ref{skcom}).
\end{rem}

The following proposition is the analogy of Proposition\ref{always Poisson}; that is, a 2-section $\pi$ on a Jacobi algebroid $(A,\phi_0)$ with a compatible cometric on $(A,\phi_0)$ is always a Jacobi structure on $(A,\phi_0)$.

\begin{prop}
Let $(A,\phi_0)$ be a Jacobi algebroid over $M$, $\pi $ a 2-section on $A$ and $g^*$ a pseudo-Riemannian metric on $A^*$. If a pair $(\pi,g^*)$ is compatible on $(A,\phi_0)$, then 
$[\pi,\pi]_{A,\phi_0}=0$.
\end{prop}

\begin{proof}
By the definition (\ref{phi-bracket}) of $\phi_0$-Schouten bracket $[\cdot,\cdot]_{A,\phi_0}$ on $\Gamma (\Lambda^*A)$, we have
\begin{align}
    [\pi,\pi]_{A,\phi_0}&=[\pi,\pi]_{A}+\pi\wedge \iota_{\phi_0}\pi+\iota_{\phi_0}\pi\wedge \pi\nonumber\\
                        &=d_\pi\pi+2\pi^\sharp\phi_0\wedge\pi\nonumber\\
                        &=d_\pi\pi-2X_0\wedge \pi.\label{Jacobi tashikame tochu-shiki}
\end{align}
By the fact that for any $\alpha$ and $\beta$ in $\Gamma(A^*)$,
\begin{align*}
    [\alpha, \beta]_{\pi,\phi_0}=[\alpha,\beta]_\pi+\langle X_0,\alpha\rangle\beta-\langle X_0,\beta\rangle\alpha-\pi(\alpha,\beta)\phi_0
\end{align*}
and the property that $D^{\pi,\phi_0}$ is torsion-free, we obtain for any $\alpha, \beta$ and $\gamma$ in $\Gamma(A^*)$,
\begin{align*}
    (d_\pi\pi)(\alpha, \beta, \gamma)=\sum_{\mathrm{Cycl}\,(\alpha,\beta,\gamma)}^{}(D_\alpha^{\pi,\phi_0}\pi)(\beta, \gamma)+3(X_0\wedge\pi)(\alpha, \beta, \gamma),
\end{align*}
where $\sum_{\mathrm{Cycl}\,(\alpha,\beta,\gamma)}^{}$ means the sum of the cyclic permutations of $\alpha, \beta$ and $\gamma$. Therefore by (\ref{Jacobi tashikame tochu-shiki}), we compute for any $\alpha, \beta$ and $\gamma$ in $\Gamma(A^*)$,
\begin{align*}
    [\pi,\pi]_{A,\phi_0}(\alpha, \beta, \gamma)&=(d_\pi\pi-2X_0\wedge \pi)(\alpha, \beta, \gamma)\\          &=\sum_{\mathrm{Cycl}\,(\alpha,\beta,\gamma)}^{}(D_\alpha^{\pi,\phi_0}\pi)(\beta, \gamma)+(X_0\wedge\pi)(\alpha, \beta, \gamma)\\
    &=\sum_{\mathrm{Cycl}\,(\alpha,\beta,\gamma)}^{}\left((D_\alpha^{\pi,\phi_0}\pi)(\beta, \gamma)+\frac{1}{2}((X_0\otimes \pi)(\beta,\gamma,\alpha)\right.\\
    &\quad \phantom{\frac{1}{2}}+(X_0\otimes \pi)(\gamma,\alpha,\beta)+g^*(\alpha,\beta)\pi((g^*)^{\flat-1}(X_0),\gamma) \\
    &\qquad \qquad \qquad \qquad \quad \ \ \  \left.\phantom{\frac{1}{2}}-g^*(\alpha,\gamma)\pi((g^*)^{\flat-1}(X_0),\beta))\right).
\end{align*}
Since $(\pi,g^*)$ is compatible on $(A,\phi_0)$, the consequence holds.
\end{proof}

The compatibility with a cometric is ``preserved'' by the Poissonization.  To be precise, the following theorem holds. 
\begin{theorem}\label{main thm}
Let $(A,\phi_0)$ be a Jacobi algebroid over $M$, $\pi \in \Gamma (\Lambda ^2A)$ a Jacobi structure on $(A,\phi_0)$ and $g^*$ a pseudo-Riemannian metric on $A^*$. For the Poissonization $\tilde{\pi}:=e^{-t}\pi \in \Gamma (\Lambda ^2\tilde{A})$ of $\pi$ and a pseudo-Riemannian metric $\tilde{g}^*:=e^{-t}g^*$ on $\tilde{A}^*$, a pair $(\pi,g^*)$ is compatible on $(A,\phi_0)$ if and only if $(\tilde{\pi},\tilde{g^*})$ is compatible on $\tilde{A}_{{\phi}_{0}}^{-}$.
\end{theorem}

\begin{proof}
It is easy to confirm that the Levi-Civita connection $\tilde{D}$ of $\tilde{g}^*$ on $(\tilde{A}_{{\phi}_{0}}^{-})^*_{\tilde{\pi}}=(\tilde{A}^*,[\cdot,\cdot\bar{]}_{\tilde{\pi}}^{\phi_0},\bar{\rho}_{\tilde{\pi}}^{\phi_0})$, where
\begin{align}
    [\tilde{\alpha},\tilde{\beta}\bar{]}_{\tilde{\pi}}^{\phi_0}&:=\overline{\mathcal{L}^{A}}_{\!\!\tilde{\pi}\tilde{\alpha}}^{\phi_{0}}\tilde{\beta}-\overline{\mathcal{L}^{A}}_{\!\!\tilde{\pi}\tilde{\beta}}^{\phi_{0}}\tilde{\alpha}-\bar{d}_{A}^{\phi_{0}}\langle\tilde{\pi}^\sharp\tilde{\alpha},\tilde{\beta}\rangle\quad (\forall\tilde{\alpha}, \tilde{\beta}\in\Gamma(\tilde{A}^*)),\\
    \bar{\rho}_{\tilde{\pi}}^{\phi_0}&:=\bar{\rho}_{A}^{\phi_0}\circ\tilde{\pi}^\sharp,\label{gotegote anchor}
\end{align}
can be written explicitly as follows: for any $\tilde{\alpha}$ and $\tilde{\beta}$ in $\Gamma(\tilde{A}^*)$,
\begin{align}
    \tilde{D}_{\tilde{\alpha}}\tilde{\beta}=e^{-t}&\left(D_{\tilde{\alpha}}^{\pi,\phi_0}\tilde{\beta}+\langle X_0,\tilde{\alpha}\rangle\left(\frac{\partial\tilde{\beta}}{\partial t}-\frac{1}{2}\tilde{\beta}\right)+\frac{1}{2}\langle X_0,\tilde{\beta}\rangle\tilde{\alpha}\right.\nonumber\\
&\qquad \qquad \qquad \qquad \qquad \left.\phantom{\frac{\partial\tilde{\beta}}{\partial t}}-\frac{1}{2}g^*(\tilde{\alpha},\tilde{\beta})(g^*)^{\flat-1}(X_0)\right).\label{explicitly Levi-Civita}
\end{align}
For any $\tilde{\alpha}, \tilde{\beta}$ and $\tilde{\gamma}$ in $\Gamma(\tilde{A}^*)$,
\begin{align*}
    (\tilde{D}\tilde{\pi})(\tilde{\alpha},\tilde{\beta},\tilde{\gamma})&=(\tilde{D}_{\tilde{\alpha}}\tilde{\pi})(\tilde{\beta},\tilde{\gamma})\\
    &=\bar{\rho}_{\tilde{\pi}}^{\phi_0}(\tilde{\alpha})(\tilde{\pi}(\tilde{\beta},\tilde{\gamma}))-\tilde{\pi}(\tilde{D}_{\tilde{\alpha}}\tilde{\beta},\tilde{\gamma})-\tilde{\pi}(\tilde{\beta},\tilde{D}_{\tilde{\alpha}}\tilde{\gamma}).
\end{align*}
Here by using (\ref{gotegote anchor}) and (\ref{explicitly Levi-Civita}), we have
\begin{align*}
    \bar{\rho}_{\tilde{\pi}}^{\phi_0}(\tilde{\alpha})(\tilde{\pi}(\tilde{\beta},\tilde{\gamma}))
    &=e^{-2t}\left(\rho_\pi(\tilde{\alpha})(\pi(\tilde{\beta},\tilde{\gamma}))-\langle X_0, \tilde{\alpha}\rangle\pi(\tilde{\beta},\tilde{\gamma})\phantom{\frac{\partial\tilde{\beta}}{\partial t}}\right.\\
    &\qquad \qquad \qquad \left.+\langle X_0, \tilde{\alpha}\rangle\pi\left(\frac{\partial\tilde{\beta}}{\partial t},\tilde{\gamma}\right)+\langle X_0, \tilde{\alpha}\rangle\pi\left(\tilde{\beta},\frac{\partial\tilde{\gamma}}{\partial t}\right)\right),\\
    \tilde{\pi}(\tilde{D}_{\tilde{\alpha}}\tilde{\beta},\tilde{\gamma})
    &=e^{-2t}\left(\pi(D_{\tilde{\alpha}}^{\pi,\phi_0}\tilde{\beta},\tilde{\gamma})+\langle X_0,\tilde{\alpha}\rangle\pi\left(\frac{\partial\tilde{\beta}}{\partial t},\tilde{\gamma}\right)\right.\\
    &\qquad \qquad \qquad -\frac{1}{2}\langle X_0,\tilde{\alpha}\rangle\pi(\tilde{\beta},\tilde{\gamma})+\frac{1}{2}\langle X_0,\tilde{\beta}\rangle\pi(\tilde{\alpha},\tilde{\gamma})\\
    &\qquad \qquad \qquad \qquad \qquad \qquad  \left.-\frac{1}{2}g^*(\tilde{\alpha},\tilde{\beta})\pi((g^*)^{\flat-1}(X_0),\tilde{\gamma})\right),\\
    \tilde{\pi}(\tilde{\beta},\tilde{D}_{\tilde{\alpha}}\tilde{\gamma})
    &=e^{-2t}\left(\pi(\tilde{\beta},D_{\tilde{\alpha}}^{\pi,\phi_0}\tilde{\gamma})+\langle X_0,\tilde{\alpha}\rangle\pi\left(\tilde{\beta},\frac{\partial\tilde{\gamma}}{\partial t}\right)\right.\\
    &\qquad \qquad \qquad -\frac{1}{2}\langle X_0,\tilde{\alpha}\rangle\pi(\tilde{\beta},\tilde{\gamma})+\frac{1}{2}\langle X_0,\tilde{\gamma}\rangle\pi(\tilde{\beta},\tilde{\alpha})\\
    &\qquad \qquad \qquad \qquad \qquad \qquad \left.-\frac{1}{2}g^*(\tilde{\alpha},\tilde{\gamma})\pi(\tilde{\beta},(g^*)^{\flat-1}(X_0))\right).
\end{align*}
It thus follows that
\begin{align*}
    (\tilde{D}\tilde{\pi})(\tilde{\alpha},\tilde{\beta},\tilde{\gamma})
    &=e^{-2t}\left(\rho_\pi(\tilde{\alpha})(\pi(\tilde{\beta},\tilde{\gamma}))-\pi(D_{\tilde{\alpha}}^{\pi,\phi_0}\tilde{\beta},\tilde{\gamma})-\pi(\tilde{\beta},D_{\tilde{\alpha}}^{\pi,\phi_0}\tilde{\gamma})\right.\\
    &\qquad \qquad -\frac{1}{2}\langle X_0,\tilde{\beta}\rangle\pi(\tilde{\alpha},\tilde{\gamma})+\frac{1}{2}g^*(\tilde{\alpha},\tilde{\beta})\pi((g^*)^{\flat-1}(X_0),\tilde{\gamma})\\
    &\qquad \qquad \left.-\frac{1}{2}\langle X_0,\tilde{\gamma}\rangle\pi(\tilde{\beta},\tilde{\alpha})+\frac{1}{2}g^*(\tilde{\alpha},\tilde{\gamma})\pi(\tilde{\beta},(g^*)^{\flat-1}(X_0))\right)\\
    &=e^{-2t}\left((D_{\tilde{\alpha}}^{\pi,\phi_0}\pi)(\tilde{\beta},\tilde{\gamma})+\frac{1}{2}\left((X_0\otimes\pi)(\tilde{\beta},\tilde{\gamma},\tilde{\alpha})\right.\right.\\
    &\qquad \qquad +(X_0\otimes\pi)(\tilde{\gamma},\tilde{\alpha},\tilde{\beta})+g^*(\tilde{\alpha},\tilde{\beta})\pi((g^*)^{\flat-1}(X_0),\tilde{\gamma})\\
    &\qquad \qquad \qquad \qquad \qquad \qquad \quad  \left.\left.-g^*(\tilde{\alpha},\tilde{\gamma})\pi((g^*)^{\flat-1}(X_0),\tilde{\beta})\right)\right).
\end{align*}
By regarding $\tilde{\alpha}, \tilde{\beta}$ and $\tilde{\gamma}$ in $\Gamma(\tilde{A}^*)$ as curves in $\Gamma(A^*)$, the conclusion follows immediately.
\end{proof}


\subsection{Contact pseudo-metric structures and Sasakian pseudo-metric structures}\label{contact to sasaki no rei}
In this subsection, 
we prove that for a contact pseudo-metric structure on a manifold, satisfying the compatibility condition is equivalent to being a Sasakian pseudo-metric structure.
This means that a Jacobi structure with a compatible cometric is considered as a generalization of a Sasakian pseudo-metric structure.
Before that, we recall the definitions of almost contact manifolds and Sasakian pseudo-metric manifolds in short. See \cite{CalPer} for details.

An {\it almost contact structure} on a $(2n+1)$-dimensional manifold $M$ is a triple $(\phi, \xi, \eta)$ of
a $(1,1)$-tensor field $\phi$ on $M$, a vector field $\xi$ on $M$ and a 1-form $\eta$ on $M$ satisfying
\begin{equation*}
{\phi}^2=-\rm{id}+\eta \otimes \xi, \quad \langle\eta,\xi\rangle=1.
\end{equation*} 
For an almost contact structure $(\phi,\xi,\eta)$ on $M^{2n+1}$, define an almost complex structure $J$ on $M\times \mathbb{R}$ by 
\begin{equation*}
J\left(X+f\frac{d}{dt}\right)=\phi X+f\xi -\langle\eta,X\rangle\frac{d}{dt}
\end{equation*}
for any $X \in \mathfrak{X}(M\times \mathbb{R})$ tangent to $M$ and $f \in C^{\infty}(M\times \mathbb{R})$, where $t$ is the standard coordinate on $\mathbb{R}$.
An almost contact structure $(\phi,\xi,\eta)$ on $M$ is called {\it normal} if this almost complex structure $J$ on $M\times \mathbb{R}$ is integrable.

A quadruple $(\phi, \xi, \eta,g)$ of an almost contact structure $(\phi,\xi,\eta)$ on $M^{2n+1}$ and a pseudo-Riemannian metric $g$ with signature $(p,q)$ on $M$ is called 
an {\it almost contact pseudo-metric structure} on $M$ if
\begin{equation*}
g(\phi X_1,\phi X_2)=g(X_1,X_2)-\varepsilon\eta (X_1)\eta (X_2)
\end{equation*}
for any $X_1,X_2 \in \mathfrak{X}(M)$, where $\varepsilon:=(-1)^q$.
Furthermore, if $\eta$ is a contact form and it satisfies for any $X_1,X_2 \in \mathfrak{X}(M)$
\begin{equation*}
g(\phi X_1,X_2)=(d\eta )(X_1,X_2),
\end{equation*}
then $(\phi, \xi, \eta,g)$ is called a {\it contact pseudo-metric structure} on $M$. In particular, a contact pseudo-metric structure $(\phi, \xi, \eta,g)$ is called a {\it contact metric structure} if $g$ is a Riemannian metric.

A normal contact pseudo-metric structure is called a {\it Sasakian pseudo-metric structure}. In particular, we call a normal contact metric structure a {\it Sasakian structure} simply. 
It is known that an almost contact pseudo-metric structure $(\phi,\xi,\eta,g)$ on $M$ is a Sasakian pseudo-metric structure if and only if 
\begin{equation*}
(\nabla_{X_1}\phi)X_2=-\frac{1}{2}g(X_1,X_2)\xi+\frac{1}{2}\varepsilon\langle\eta,X_2\rangle X_1
\end{equation*}
for any $X_1,X_2 \in \mathfrak{X}(M)$.
Moreover, for a Sasakian pseudo-metric structure $(\phi,\xi,\eta,g)$ on $M$, 
\begin{equation*}
\nabla_{X}\xi=\frac{1}{2}\varepsilon\phi X,\quad \mathcal{L}_\xi g=0
\end{equation*}
hold for any $X \in \mathfrak{X}(M)$.

\begin{rem}
In contact geometry, a wedge product $(\alpha\wedge\beta)(X_1,\dots,X_{k+l})$ for any $\alpha$ in $\Omega^k(M), \beta$ in $\Omega^l(M)$ and $X_i$ in $\mathcal{X}(M)$ for $i=1,\dots,k+l$ is often defined as $\sum_{\sigma\in S_{k+l}}\mathrm{sgn}\ \!\sigma \alpha(X_{\sigma(1)},\dots,X_{\sigma(k)})\beta(X_{\sigma(k+1)},\dots,X_{\sigma(k+l)})$ multiplied by $\frac{1}{(k+l)!}$. However, in this paper, we adopt that multiplied by $\frac{1}{k!l!}$, which is often used in the context of Lie algebroid theory. These differences cause the various formulas to change slightly. See \cite{Bla} for the differences.
\end{rem}

\begin{theorem}\label{Sasakian pseudo-metric}
Let $(M,\phi,\xi,\eta,g)$ be a contact pseudo-metric manifold and $(\Lambda,E)$ the corresponding Jacobi structure to the contact form $\varepsilon\eta$ on $M$.  
Let $G$ be a pseudo-Riemannian metric on $TM\oplus \mathbb{R}$ defined by 
\begin{equation*}
G((X_1,f),(X_2,h)):=g(X_1,X_2)+\varepsilon fh
\end{equation*}
and $G^*$ the dual metric of $G$ on $T^*M\oplus \mathbb{R}$ with respect to $(\Lambda,E)$. 
Then the pair $((\Lambda,E),G^*)$ is compatible on $(TM\oplus \mathbb{R}, (0,1))$ if and only if $(\phi,\xi,\eta,g)$ is a Sasakian pseudo-metric structure on $M$.
\end{theorem}
\begin{proof}
From a direct calculation, the condition that $((\Lambda,E),G^*)$ is compatible on $(TM\oplus \mathbb{R}, (0,1))$ is described as follows:
\begin{align*}
&\varepsilon g\left(\phantom{\frac{1}{2}}\!\!\!\!(\nabla_{X_1}\phi)X_2+\frac{1}{2}g(X_1,X_2)\xi-\frac{1}{2}\varepsilon\langle\eta,X_2\rangle X_1,X_3\right)\\
&\qquad +h_2g\left(\nabla_{X_1}\xi-\frac{1}{2}\varepsilon\phi X_1,X_3\right)-\frac{1}{2}\varepsilon h_3(\mathcal{L}_\xi g)(X_1,X_2)=0
\end{align*}
for any $X_1,X_2, X_3 \in \mathfrak{X}(M)$ and any $h_2,h_3 \in C^{\infty}(M)$.

The condition above is actually equivalent to that $(\phi,\xi,\eta,g)$ is a Sasakian pseudo-metric structure on $M$.
\end{proof}

Because of Theorem \ref{Sasakian pseudo-metric}, a Jacobi structure with a compatible cometric is considered as a generalization of a Sasakian pseudo-metric structure.

From Theorem \ref{main thm}, the condition that $((\Lambda,E),G^*)$ is compatible on $(TM\oplus \mathbb{R}, (0,1))$ is equivalent to 
that the pair $(e^{-t}\pi, e^{-t}G^*)$ is compatible on the Lie algebroid $T(M\times \mathbb{R})$, 
where $\pi \in \Gamma (\Lambda ^2(TM\oplus \mathbb{R}))$ is the Jacobi structure on $TM\oplus \mathbb{R}$ corresponding to $(\Lambda,E)$.  
The following well known fact is recovered from this observation and Theorem \ref{Sasakian pseudo-metric} immediately.

\begin{corollary}
A quadruple $(\phi,\xi,\eta,g)$ is a Sasakian structure on $M$ if and only if $(d(e^t\eta), J, e^tG)$ is a K\"{a}hler structure on $M\times \mathbb{R}$.
\end{corollary}



\end{document}